\documentclass[12pt]{article}

\usepackage{amsmath,amsthm,amsfonts,amssymb,amscd}

\allowdisplaybreaks[4]

\newtheorem {lemma}{Lemma}[section]
\newtheorem {theorem} {Theorem}[section]

\newtheorem {claim}{Claim}[section]
\newtheorem {problem}{Problem}[section]

\usepackage{tikz}

\begin{document}

\title{Eigenvalue conditions implying edge-disjoint spanning trees and a forest with constraints}

\author{Jin Cai\footnote{Email: jincai@m.scnu.edu.cn}, Bo Zhou\footnote{Email: zhoubo@m.scnu.edu.cn}\\
School of Mathematical Sciences, South China Normal University\\
Guangzhou 510631, P.R. China}

\date{}
\maketitle

\begin{abstract}
Let $G$ be a nontrivial graph with minimum degree $\delta$ and $k$ an integer with $k\ge 1$.
In the literature, there are eigenvalue conditions  that imply $G$  contains $k$ edge-disjoint spanning trees.
We give eigenvalue conditions  that imply $G$  contains $k$ edge-disjoint spanning trees and another forest $F$ with $|E(F)|>\frac{\delta-1}{\delta}(|V(G)|-1)$, and if $F$ is not a spanning tree, then $F$ has a component with at least $\delta$ edges.
\end{abstract}

\section{Introduction}
In this paper, we consider finite, undirected and simple graphs.
As usual, $K_n$ denotes a complete graph of order $n$.
For a connected graph $G$, let $\tau (G)$  be  the maximum number of edge-disjoint spanning trees in  $G$, which is also known as the spanning-tree packing number, see \cite{GuLiu,LaiLi,Le,Pa}. By definition, $\tau(K_1)=\infty$, and $\tau(G)=0$ if $G$ is disconnected.

The eigenvalues of a graph are the eigenvalues of
its  adjacency matrix.  Since the adjacency matrix of a graph is a real symmetric matrix,  every eigenvalue of a graph is real.
For a graph $G$ of order $n$, let $\lambda_i(G)$ be the $i$-th largest eigenvalue of $G$ with $i=1,\dots, n$.  $\lambda_1(G)$ is also known as the spectral radius of $G$.

Seymour proposed the following problem (in private communication to Cioab\v a) relating  $\tau(G)$ and eigenvalues of $G$.

\begin{problem}\label{q1} \cite{Cio}
Let $G$ be a nontrivial graph. Determine the relationship between $\tau(G)$ and eigenvalues of $G$.
\end{problem}

Motivated by Problem \ref{q1}, Cioab\v{a} and  Wong \cite{Cio} established the first a few results on Problem \ref{q1}. They
proposed a conjecture: Let $k$ be an integer with $k\ge 2$ and $G$ be a $d$-regular graph with $d\ge 2k$. If $\lambda_2(G)\le d-\frac{2k-1}{d+1}$, then $\tau (G)\ge k$. This was then generalized  by
Liu, Hong, Gu and Lai \cite{LHGL} claiming that it holds for a graph $G$ with minimum degree $\delta\ge 2k$, which was confirmed by Gu et al. \cite{Gu}. Further results may be found in  \cite{COP} and references therein.
%
%
%
We state two  typical works.

For positive integers $n$ and $s$ and a nonnegative integer $k$ with $n\ge s+k$, let $G\cong B_{n,s}^{k}$ be a graph obtained from disjoint $K_{s}$ and $K_{n-s}$ by adding $k$ edges joining a vertex in $K_{s}$ and $k$ vertices in $K_{n-s}$. In particular, $B_{n,s}^{0}=K_s\cup K_{n-s}$.

\begin{theorem}\label{spectral}\cite{Fan1}
Let $k$ be an integer with $k\ge 2$, and let $G$ be a connected graph with minimum degree $\delta\ge 2k$ and order $n\ge 2\delta+3$. If $\lambda_1(G)\ge \lambda_1(B_{n,\delta+1}^{k-1})$,  then $\tau(G)\ge k$
unless  $G\cong B_{n,\delta+1}^{k-1}$.
\end{theorem}

Generalizing sufficient eigenvalue conditions for a $d$-regular graph $G$ with $\tau (G)\ge k$
in \cite{Cio}, Lui et al. \cite{LHGL} established the following theorem on graphs with fixed minimum degree.

\begin{theorem}\label{min}\cite{LHGL}
Let $k$ be an integer with $k\ge 2$, and let $G$ be a graph with minimum degree $\delta\ge 2k$. If $\lambda_2(G)<\delta-\frac{2k-1}{\delta+1}$,  then $\tau(G)\ge k$.
\end{theorem}

To obtain these results, the Tree Packing Theorem due to Tutte \cite{Tutte} and Nash-Williams \cite{NW}
over 60 years ago is needed.
For a graph $G$ we denote by $V(G)$ the vertex set and $E(G)$ the edge set.
For vertex disjoint subset $X, Y\subset V(G)$, $E(X,Y)$ denotes the set of edges of $G$  with one end vertex in $X$ and the other end vertex in $Y$, and let $e(X,Y)=|E(X,Y)|$.
For any partition $\mathcal{P}$ of $V(G)$ of a nontrivial graph $G$, the size $|\mathcal{P}|$ of partition $\mathcal{P}$ is the number parts. The fractional packing number $\nu_f(G)$ is defined by
\[
\nu_f(G)=\min_{|\mathcal{P}|\ge 2} \frac{\sum_{1\le i< j \le |\mathcal{P}|}e(V_i,V_j)}{|\mathcal{P}|-1}.
\]

\begin{theorem}\label{NW}\cite{NW,Tutte} [Tree Packing Theorem] For a nontrivial graph $G$ and a  nonnegative integer $k$,
$\tau(G)\ge k$ if and only if $\nu_f(G)\ge k$.
\end{theorem}

Fan et al. \cite{Fan} established an extension of Theorem \ref{NW} stating that
for a graph $G$ of order $n$, if $\nu_f(G)=k+\varepsilon$ with $0\le \varepsilon <1$, then $\tau(G)\ge k$, and  apart from $k$ edge-disjoint spanning trees, there is  another forest with at least $\varepsilon(n-1)$ edges in $G$. Recently,
Fang and Yang \cite{Fang} gave  a structural explanation for the fractional part $\varepsilon$.

For a nonnegative integer $k$ and a positive integer $d$, a graph $G$ is said to have property $P(k, d)$ if $G$ satisfies the following three conditions:
\begin{enumerate}

\item[(a)]
$\tau(G)\ge k$,

\item[(b)] apart from $k$ edge-disjoint spanning trees, there is another forest $F$ with $|E(F)|>\frac{d-1}{d}(|V(G)|-1)$,

\item[(c)]
if $F$ is not a spanning tree, then $F$ has a component with at least $d$ edges.
\end{enumerate}

Note that any graph $G$ with minimum degree $\delta\ge 1$  has property $P(0, \delta)$. This is obvious
if $G$ is connected. Suppose that $G$ is a disconnected graph with order $n$ and minimum degree $\delta\ge 1$. Then every component has at least $\delta+1$ vertices, so it has at most $\frac{n}{\delta+1}$ components. It follows that
$G$ has a spanning forest $F$ such that $|E(F)|\ge n-\frac{n}{\delta+1}>\frac{\delta-1}{\delta}(n-1)$.  Evidently, each component of $F$  has at least $\delta$ edges. Thus,  $G$ has property $P(0, \delta)$.

\begin{theorem}\label{Fang}\cite{Fang}
For positive integers $k$ and $d$, and a nontrivial graph $G$, if $\nu_f(G)>k+\frac{d-1}{d}$, then
$G$ has property $P(k, d)$.
\end{theorem}

Motivated by the above works, we investigate the following problem.

\begin{problem} \label{q2} Let $k$ be a positive integer, and let $G$ be a nontrivial graph with minimum degree $\delta$.
What eigenvalue conditions imply that $G$ has property $P(k, \delta)$?
\end{problem}

The main results are listed as below.

\begin{theorem}\label{s2}
Let $k$ be a positive integer, and let $G$ be a graph with minimum degree $\delta\ge 2k+2$ and order $n\ge 2\delta+3$.
If $\lambda_1(G)\ge \lambda_1(B_{n,\delta+1}^{k-1})$, then
$G$ has property $P(k, \delta)$ unless  $G\cong B_{n,\delta+1}^{k-1}$.
\end{theorem}

We remark that the lower bound $2k+2$ on $\delta$ in Theorem \ref{s2} can not be lowered to $2k$ as in Theorem \ref{spectral} or even $2k+1$, generally. We give such examples.

(i) Let $H_1$ be the graph in Fig.~\ref{F123}. By a direct calculation, we have $5.1919=\lambda_1(H_1)>\lambda_1(B_{11,5}^{1})=5.0561$.
Note that $\delta(H_1)=4=2\times 2$, $\tau(H_1)=2$, and apart from two edge-disjoint spanning trees (whose edges are displayed as bold and thin lines), there is another forest $F$ (whose edges are displayed as dashed lines) with $8=|E(F)|>\frac{3}{4}\times 10=\frac{15}{2}$. However,  $F$ is not a spanning tree, and $F$ has no component with at least $4$ edges, so $H_1$ does not have property $P(2, 4)$.

\begin{figure}[htbp]
\centering
\begin{tikzpicture}[scale=1]
\draw  [dashed, blue](-2,2)--(-1,2)--(-1,1)--(-2,1);
\filldraw [black] (-2,2) circle (2pt);
\filldraw [black] (-1,2) circle (2pt);
\filldraw [black] (-1,1) circle (2pt);
\filldraw [black] (-2,1) circle (2pt);
\draw  [dashed, blue](1,2)--(1,1)--(2,2)--(2,1);
\filldraw [black] (1,2) circle (2pt);
\filldraw [black] (1,1) circle (2pt);
\filldraw [black] (2,2) circle (2pt);
\filldraw [black] (2,1) circle (2pt);
\draw  [dashed, blue](-1,-0.5)--(1,-0.5)--(0,-1);
\filldraw [black] (-1,-0.5) circle (2pt);
\filldraw [black] (1,-0.5) circle (2pt);
\filldraw [black] (0,-1) circle (2pt);
\draw  [line width=1.3pt, black](-1,2)--(-2,1)--(-2,2)--(-1,1);
\draw  [line width=1.3pt,black](-2,1)--(-1,-0.5)--(0,-1)--(2,1)--(1,1)--(1,-0.5);
\draw  [line width=1.3pt,black](2,1)--(1,2)--(2,2);
\draw  [black](-1,2)--(1,2)--(0,-1)--(2,2)--(1,-0.5)--(2,1);
\draw  [black](1,1)--(-1,1)--(-1,-0.5);
\draw  [black](-1,1)--(-1,-0.5)--(-2,2);
\draw  [black](-2,1)--(0,-1);
\draw  [black](-1,1)--(1,-0.5);
\end{tikzpicture}
\caption{Graph $H_1$.}
\label{F123}
\end{figure}
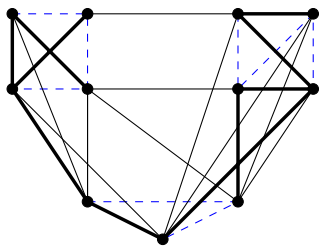

(ii)  Let $H_2$ be the graph obtained from $K_{16}\cup K_{17}$ by removing two independent edge $uv, xy$ from $K_{16}$ and adding edges
$uw, vw, xw, yw$ and three other edges connecting three vertices of $K_{16}$ to some vertex $w$ of $K_{17}$. By a  calculation, we have $16.1578=\lambda_1(H_2)>\lambda_1(B_{33,16}^{6})=15.1645$.
Note that $\delta(H_2)=15=2\times 7+1$ and that $\tau(H)\ge 7$. Apart from $7$ edge-disjoint spanning trees, any forest has at most $13+16=29$ edges as there are only $13$ edges left in $K_{16}$ after removing two independent edges and the edges of $7$ spanning trees. However, $29<\frac{14}{15}\times 32=\frac{448}{15}$, so $H_2$  does not have property $P(7, 15)$.

\begin{theorem}\label{s1}
Let $k$ be a positive integer, and let $G$ be a graph with minimum degree $\delta\ge 2k+2$.
Suppose that
\[
\lambda_2(G)<\delta-\frac{2\left(k+\frac{\delta-1}{\delta}\right)}{\delta+1},
\]
then $G$ has property $P(k, \delta)$.
\end{theorem}

To attain the conclusion  Theorem \ref{s1},  the lower bound $2k+2$ on  $\delta$  can not be lowered to $2k$ as in Theorem \ref{min} or even $2k+1$,  generally.

(i) For $k\ge 1$, $\lambda_2(K_{2k+1})=-1$, which is less than the bound in Theorem \ref{s1}. Note that $\delta(K_{2k+1})=2k$ and $\tau(K_{2k+1})=k$ \cite{Pa}. However, apart from
$k$ edge-disjoint spanning trees, there are $k\le 2k-1$ edges, so $K_{2k+1}$ does not have property $P(k, 2k)$.

(ii) Let $H$ be the Petersen graph. Note that $\lambda_2(H)=1<3-\frac{2\times \left(1+\frac{2}{3}\right)}{3+1}=\frac{13}{6}$ and $\tau (H)=1$. Apart from one spanning tree, there are $6$ edges.  But
$6=\frac{3-1}{3}(10-1)$, so $H$ does not have property $P(1,3)$.
For $k\ge 2$, $\lambda_2(K_{2k+1,2k+1})=0$, which is less than the bound of Theorem \ref{s1}. Note that $\delta(K_{2k+1,2k+1})=2k+1$ and $\tau(K_{2k+1,2k+1})=k$ \cite{Pa}. However, apart from
$k$ edge-disjoint spanning trees, there are $3k+1$ edges. As $3k+1< \frac{2k}{2k+1}(4k+1)$,  $K_{2k+1,2k+1}$ does not have property $P(k, 2k+1)$.

The rest of the paper is organized as follows. In Section 2 we give lemmas that will be used. Theorems \ref{s2} are \ref{s1} are proved in Sections 3 and 4, respectively.

\section{Preliminaries}

Let $G$ be a graph. For a vertex $v$ of $G$, we denote by $d_G(v)$ the degree of $v$ in $G$.
Denote by $G[S]$ the subgraph of $G$ induced by $S$ if $\emptyset\ne S\subseteq V(G)$,
and $G-E_1$ the graph with vertex set $V(G)$ and edge set
$E(G)\setminus E_1$ if $E_1\subseteq E(G)$, and in particular, we write $G-f$ for $G-\{f\}$ when $f\in E(G)$.
Given two graphs $G$ and $H$, let $G\cup H$ denote the disjoint union of $G$ and $H$.

If all the eigenvalues of an $n\times n$ matrix $B$
are real, then we denote them  by
$\lambda_1(B),\dots, \lambda_n(B)$ with $\lambda_1(B)\ge \dots\ge \lambda_n(B)$.
For an $n$-vertex graph $G$, the adjacency matrix of $G$ is the $n\times n$ matrix
$A(G)=(a_{uv})_{u,v\in V(G)}$, where $a_{uv}=1$ if $uv\in E(G)$ and $a_{uv}=0$ otherwise.
Evidently, $\lambda_i(G)=\lambda_i(A(G))$ for $i=1,\dots, n$.
Note that $\sum_{i=1}^{n}\lambda_i(G)=0$.

The following sharp upper bound on the spectral radius was obtained by Hong et al. \cite{Hong} and Nikiforov \cite{Nikiforov}.
\begin{lemma}\label{upper}
Let $G$ be a graph on $n$ vertices and $m$ edges with minimum degree $\delta\ge 1$. Then
\[
\lambda_1(G)\le \frac{\delta-1}{2}+\sqrt{2m-n\delta +\frac{(\delta+1)^2}{4}}.
\]
\end{lemma}

By the well-known Perron-Frobenius theorem, we can easily deduce the following lemma.
\begin{lemma}\label{subgraph}
If $H$ is a subgraph of a connected graph $G$, then
\[
\lambda_1(H)\le \lambda_1(G)
\]
with equality if and only if $H\cong G$.
\end{lemma}

Consider two sequences of real numbers: $\eta_1\ge \eta_2\ge \dots \ge \eta_n$ and $\mu_1\ge \mu_2\ge \dots \ge \mu_m$ with $m<n$. The second sequence is said to interlace the first one whenever
\[
\eta_i\ge \mu_i\ge \eta_{n-m+i} ~ for ~i=1,2,\dots, m.
\]


The following lemma is the well-known Cauchy Interlacing Theorem,  see \cite{Bro, God}.

\begin{lemma}\label{inter}
Let $A$ be a real symmetric matrix and $B$ be a principal submatrix of $A$. Then the eigenvalues of $B$ interlace the eigenvalues of $A$.
\end{lemma}

Suppose that $G$ is a graph and $V(G)$ is partitioned as $V_1\cup \dots\cup V_m$. For $1\le i<j\le m$, denote by $A_{ij}$  the submatrix of $A(G)$ with rows corresponding to vertices in $V_i$ and columns corresponding to vertices in $V_j$. The quotient matrix of $A(G)$ with respect to this partition is the matrix  $B=(b_{ij})$, where $b_{ij}=\frac{1}{|V_i|}\sum_{u\in V_i}\sum_{v\in V_j}a_{uv}$.

The following lemma is a special case of Corollary 2.3 in \cite{Hae}, see also \cite{Bro, God}.

\begin{lemma}\label{quo}
For a graph $G$, if $B$ is a quotient matrix of $A(G)$, then the eigenvalues of $B$ interlace the eigenvalues of $A(G)$.
\end{lemma}

\begin{lemma}\label{U} Let $G$ be a graph with minimum degree $\delta\ge 1$. Let $\emptyset\ne U\subset V(G)$. If $e(U,V\setminus U)\le \delta-1$, then $|U|\ge \delta+1$.
\end{lemma}
\begin{proof}
If $|U|\le \delta$, then
\[
\delta |U|\le \sum_{u\in U}d_G(u)\le |U|(|U|-1)+e(U,V\setminus U)\le \delta(|U|-1)+\delta-1=\delta |U|-1,
\]
a contradiction.
\end{proof}

\section{Spectral radius condition: Proof of Theorem \ref{s2}}

For positive integers $n$, $k$, and $s$ with $n\ge s+k$ and $k\ge 2$, let  $\mathcal{G}_{n,s}^{k}$ be the set of graphs obtained from $K_{s}\cup K_{n-s}$ by adding $k$ edges between $K_{s}$ and $K_{n-s}$. It is evident that
$B_{n,s}^{k}\in \mathcal{G}_{n,s}^{k}$.

\begin{lemma}\label{disconnected}
Let $G$ be a disconnected graph on $n$ vertices with minimum degree $\delta$, where $n\ge 2\delta+2$.
Then $\lambda_1(G)\le n-\delta-2$ with equality if and only if $G\cong B_{n,\delta+1}^{0}$.
\end{lemma}

\begin{proof} Suppose that $G$ is a disconnected graph on $n$ vertices with minimum degree $\delta$
that maximizes the spectral radius. Then, for some component $H$ of $G$,
\[
\lambda_1(H)=\lambda_1(G)\ge \lambda_1(K_{\delta+1}\cup K_{n-\delta-1})=\max\{\delta, n-\delta-2\}=n-\delta-2.
\]
So $|V(H)|\ge \lambda_1(H)+1\ge n-\delta-1$. As the minimum degree of $G$ is $\delta$, $G$ consists of two components,
$|V(H)|=n-\delta-1$, $G-V(H)\cong K_{\delta+1}$,  and so we have by Lemma \ref{subgraph} that $H\cong K_{n-\delta-1}$. Thus $G\cong K_{\delta+1}\cup K_{n-\delta-1}=B_{n,\delta+1}^{0}$.
\end{proof}

\begin{lemma}\label{B}\cite{Fan1}
Let $G\in \mathcal{G}_{n,\delta+1}^{k-1}$ where $k\ge 2$, $n\ge 2\delta+3$, and $\delta\ge 2k$. Then $\lambda_1(G)\le \lambda_1(B_{n,\delta+1}^{k-1})$ with equality if and only if $G\cong B_{n,\delta+1}^{k-1}$.
\end{lemma}

\begin{lemma}\label{change}\cite{Fan1}
Let $G\in \mathcal{G}_{n,b}^{k-1}$ where $k\ge 2$,  $n\ge 2b$, $b\ge \delta+2$, and $\delta\ge 2k$. Then $\lambda_1(G)<\lambda_1(B_{n,\delta+1}^{k-1})$.
\end{lemma}

\begin{lemma}\label{choose}
For integers $x, y$ and $a$ with $x,y\ge a\ge 2$, ${x\choose 2}+{y\choose 2}\le {a\choose 2}+{x+y-a\choose 2}$.
\end{lemma}

\begin{proof} It is equivalent to the trivial inequality  $(x-a)(y-a)\ge 0$.
\end{proof}

\begin{lemma}\label{t1}
For positive integers $a_1,\dots, a_p$,
\[
\sum_{i=1}^p{a_i\choose 2}\le {\sum_{i=1}^p a_i-p+1\choose 2}.
\]
\end{lemma}
\begin{proof}
We prove the inequality by induction on $p$. It is trivial if $p=1$.
Suppose that $p\ge 2$, and that
$\sum_{i=1}^{p-1} {a_i\choose 2}\le {\sum_{i=1}^{p-1}a_i-(p-1)+1\choose 2}$.
Then
\begin{align*}
\sum_{i=1}^p{a_i\choose 2}&=\sum_{i=1}^{p-1} {a_i\choose 2}+{a_p\choose 2}\\
&\le {\sum_{i=1}^{p-1}a_i-(p-1)+1\choose 2}+{a_p\choose 2}.
\end{align*}
Let $a=\sum_{i=1}^{p-1}a_i-(p-1)+1$ and $b=a_p$. Then
\[
{a+b-1\choose 2}-{a\choose 2}-{b\choose 2}=(a-1)(b-1)\ge 0,
\]
so ${a\choose 2}+{b\choose 2}\le {a+b-1\choose 2}$, implying that
\[
\sum_{i=1}^p{a_i\choose 2}
\le {\sum_{i=1}^{p-1}a_i-(p-1)+1+a_p-1\choose 2}={\sum_{i=1}^p a_i-p+1\choose 2}.
\]
Thus, the desired inequality follows.
\end{proof}

\begin{lemma}\label{p=2}
Let $n\ge 2k +3$ and $k\ge 1$ be integers. Let $G$ be a graph obtained by $K_n$ deleting $k$ edges. Then $\tau(G)\ge k+1$.
\end{lemma}
\begin{proof}
Suppose that $\tau(G)<k+1$. By Theorem \ref{NW}, there exists some partition $V(G)=V_1\cup \dots \cup V_p$ of $V(G)$ with $2\le p \le n$   such that
\[
\sum_{1\le i<j\le t}e(V_i,V_j)\le (k+1)(p-1)-1.
\]
Then
\begin{align*}
e(G)&=\sum_{i=1}^{p}e(G[V_i])+\sum_{1\le i<j\le p}e(V_i,V_j)\\
&\le \sum_{i=1}^{p}{|V_i|\choose 2}+(k+1)(p-1)-1.
\end{align*}
Now, by  Lemma \ref{t1} and letting $f(p)=\frac{p^2}{2}-(n-k-\frac{1}{2})p+\frac{n(n+1)}{2}-k-2$,  we have
\begin{align*}
e(G)
&\le {n-p+1\choose 2}+(k+1)(p-1)-1\\
&=f(p).
\end{align*}
As $n\ge 2k+3$, we have $f(2)-f(n)=\frac{1}{2}n^2-(k+\frac{5}{2})n+2k+3\ge
\frac{1}{2}(2k+3)^2-(k+\frac{5}{2})(2k+3)+2k+3=0$, so $f(2)\ge f(n)$.
As $n\ge 2k+3$ and $k\ge 1$, we have
$2< n-k-\frac{1}{2}< n$.
As $2\le p\le n$, we have $f(p)\le \max\{f(2),
f(n)\}=f(2)=\frac{n^2}{2}-\frac{3n}{2}+k+1$.
So $e(G)\le \frac{n^2}{2}-\frac{3n}{2}+k+1$. However, as $n\ge 2k+3$, we have
\[
e(G)={n\choose 2}-k=\frac{n^2}{2}-\frac{n}{2}-k=\frac{n^2}{2}-\frac{3n}{2}+n-k>\frac{n^2}{2}-\frac{3n}{2}+k+1,
\]
a contradiction.
\end{proof}

We are now ready to prove Theorem \ref{s2}.

\begin{proof}[Proof of Theorem \ref{s2}]
Suppose that $G$ is a graph for which Theorem \ref{s2} is not true.  That is,
 $G$ is a graph with minimum degree $\delta\ge 2k+2$ and order $n\ge 2\delta+3$ such that
\begin{equation}\label{aaa}
\lambda_1(G)\ge \lambda_1(B_{n,\delta+1}^{k-1})
\end{equation}
but
$G$ does not have  property $P(k, \delta)$ unless  $G\cong B_{n,\delta+1}^{k-1}$.

By Theorem \ref{Fang},  $\nu_f(G)\le k+\frac{\delta-1}{\delta}$, so there exists some partition $\mathcal{P}$ of $V(G)$ into $p=|\mathcal{P}|$ subsets $V_1,\dots,V_p$ with $p\ge 2$ such that
\begin{equation} \label{11}
\sum_{1\le i<j\le p}e(V_i,V_j)\le \left(k+\frac{\delta-1}{\delta}\right)(p-1)<(k+1)(p-1).
\end{equation}

\noindent
{\bf Case 1.} $p=2$.

From \eqref{11}, we have $e(V_1,V_2)\le k$. By Lemma \ref{U}, we have $|V_1|,|V_2|\ge \delta+1$. Suppose that  $e(V_1,V_2)\le k-1$. If $k=1$,  then we have by Lemma \ref{disconnected} that
\[
\lambda_1(G)\le n-\delta-2=\lambda_1(B_{n,\delta+1}^{k-1})
\]
with equality if and only if $G\cong B_{n,\delta+1}^{k-1}$, contradicting the assumption.  Suppose that $k\ge 2$.  Assume that $|V_1|\le |V_2|$. For the cases  $|V_1|=\delta+1$ and $|V_1|\ge \delta+2$, we have
by Lemmas \ref{B} and \ref{change}, respectively, together with Lemma \ref{subgraph}, that
\[
\lambda_1(G)\le \lambda_1(B_{n,\delta+1}^{k-1})
\]
with equality if and only if $G \cong B_{n,\delta+1}^{k-1}$, contradicting the assumption again. 
This shows that  $e(V_1,V_2)=k$.

\begin{claim}\label{case1a} $e(G[V_1])+e(G[V_2])< {|V_1|\choose 2}+{|V_2|\choose 2}-k$.
\end{claim}

\begin{proof} If $e(G[V_i])< {|V_i|\choose 2}-k$ for $i=1$ or $2$, then
$e(G[V_1])+e(G[V_2])< 
{|V_1|\choose 2}+{|V_2|\choose 2}-k$,
as desired.

Suppose that  $e(G[V_i])\ge {|V_i|\choose 2}-k$ for each $i=1,2$.
For $i=1,2$, as $\delta\ge 2k+2$ and $|V_i|\ge \delta+1$, we have $|V_i|\ge 2k+3$.
By Lemma \ref{p=2}, $\tau (G[V_i])\ge k+1$ for $i=1,2$. A spanning tree of $G[V_1]$,
a spanning tree of $G[V_2]$, and one edge from $V_1$ to $V_2$ form a spanning tree of $G$.
Thus, $\tau(G)\ge k$.  Apart from the $k$ edge-disjoint spanning trees of $G$,
there is a forest $F$ of $G$ consisting a spanning tree of $G[V_1]$ and
a spanning tree of $G[V_2]$. As $n\ge 2\delta+3$, we have
\[
|E(F)|=|V_1|+|V_2|-2=n-2>\frac{\delta-1}{\delta}(n-1).
\]
As $e(V_1,V_2)=k$, $F$ is not a spanning tree. As $|V_i|\ge \delta+1$ for $i=1,2$, each component of $F$ has at least $\min\{|V_1|,|V_2|\}-1\ge \delta$ edges. So $G$ has property $P(k, \delta)$, contradicting the assumption again.
\end{proof}

By Claim \ref{case1a}, the fact that $e(V_1,V_2)=k$,  and Lemma \ref{choose}, we have
\[
e(G)< {|V_1|\choose 2}+{|V_2|\choose 2} \le {\delta+1\choose 2}+{n-\delta-1\choose 2}.
\]
So, by Lemma \ref{upper}, we have
\begin{align*}
\lambda_1(G)&< \frac{\delta-1}{2}+\sqrt{2\left({\delta+1\choose 2}+{n-\delta-1\choose 2}\right)-n\delta+\frac{(\delta+1)^2}{4}}\\
&=\frac{\delta-1}{2}+\sqrt{\left(n-\frac{3}{2}\delta-\frac{3}{2}\right)^2}\\
&=n-\delta-2.
\end{align*}
Evidently,  $K_{\delta+1}\cup K_{n-\delta-1}$ is a subgraph of $B_{n,\delta+1}^{k-1}$.
From \eqref{aaa}, we have by Lemma \ref{subgraph} that
\[
\lambda_1(G)\ge \lambda_1(B_{n,\delta+1}^{k-1})\ge\lambda_1(K_{\delta+1}\cup K_{n-\delta-1})=n-\delta-2,
\]
which is a contradiction.

\noindent
{\bf Case 2.} $p\ge 3$.

\begin{claim}\label{edges}
$e(G)<{\delta+1\choose 2}+{n-\delta-2\choose 2}+2(k+1)$.
\end{claim}
\begin{proof}

Let $r_i=e(V_i,V\setminus V_i )$ for $i=1,\dots, p$, it follows from \eqref{11} that
\begin{equation}\label{12}
\sum_{i=1}^{p}r_i<2(k+1)(p-1).
\end{equation}
If there exists at most one  part $V_j$ with $1\le j\le p$ such that $r_j\le \delta-1$, then
$r_i\ge \delta$ for all $i$ with  $1\le i\le p$ and $i\ne j$, so we have
\[
\sum_{i=1}^{p}r_i
\ge (p-1)\delta\ge 2(k+1)(p-1),
\]
which contradicts \eqref{12}.
Thus,  there exist two  parts, say $V_s, V_t$ such that $r_s, r_t\le \delta-1$. By Lemma \ref{U}, $|V_s|,|V_t|\ge \delta+1$,  so
\[
p\le n-|V_s|-|V_t|+2\le n-2\delta.
\]
By Lemmas \ref{t1} and \ref{choose},
\begin{align*}
\sum_{i=1}^p e(V_i)& \le {|V_t|\choose 2}+\sum_{i=1\atop i\ne t}^p {|V_i|\choose 2}\\
&\le  {|V_t|\choose 2}+{n-|V_t| -(p-1)+1\choose 2}\\
&\le {\delta+1\choose 2}+{n-(p+\delta-1)\choose 2}.
\end{align*}
From \eqref{11}, we have
\begin{align*}
e(G)&=\sum_{i=1}^pe(V_i)+\sum_{1\le i<j\le p}e(V_i,V_j)\\
&< {\delta+1\choose 2}+{n-(p+\delta-1)\choose 2}+(k+1)(p-1)\\
&=\frac{1}{2}p^2+(\delta+k-n+\frac{1}{2})p+\delta^2-n\delta+\frac{1}{2}n^2+\frac{1}{2}n-k-1.
\end{align*}
Let $\psi(p)=\frac{1}{2}p^2+(\delta+k-n+\frac{1}{2})p+\delta^2-n\delta+\frac{1}{2}n^2+\frac{1}{2}n-k-1$.
Since $p\le n-2\delta$ and $\delta\ge 2k+2$, we have
\[
-(\delta+k-n+\frac{1}{2})>n-2\delta\ge p\ge 3,
\]
so $\psi(p)\le \psi(3)$.
Hence Claim \ref{edges} follows.
\end{proof}

By Claim \ref{edges}, $e(G)<{\delta+1\choose 2}+{n-\delta-2)\choose 2}+2(k+1)$.
By Lemma \ref{upper}, and the facts that $n\ge 2\delta+3$ and $\delta\ge 2k+2$, we have
\begin{align*}
\lambda_1(G)&<\frac{\delta-1}{2}+
\sqrt{2\left({\delta+1\choose 2}+{n-\delta-2)\choose 2}+2(k+1)\right)-n\delta+\frac{(\delta+1)^2}{4}}\\
&=\frac{\delta-1}{2}+\sqrt{n^2 - (3\delta+5)n+\frac{9}{4}\delta^2+\frac{13}{2}\delta+4k+\frac{41}{4}}\\
&=\frac{\delta-1}{2}+\sqrt{\left(n-\frac{3}{2}\delta-\frac{3}{2}\right)^2-\left(2n-2\delta-4k-8\right)}\\
&<\frac{\delta-1}{2}+\sqrt{\left(n-\frac{3}{2}\delta-\frac{3}{2}\right)^2}\\
&=n-\delta-2.
\end{align*}
Similarly as above, we arrive at a contradiction.
\end{proof}

\section{Second largest eigenvalue condition: Proof of Theorem \ref{s1}}

If $G$ is a graph with minimum degree $\delta\ge 1$, and $\lambda_2(G)<\delta$, then $G$ is connected. This follows  from  the well known fact that minimum degree is a lower bound of the spectral radius of a graph.

\begin{proof}[Proof of Theorem \ref{s1}]
By Theorem \ref{Fang}, it suffices to show that $\nu_f(G)> k+\frac{\delta-1}{\delta}$.
So for any partition $V_1,\dots,V_p$ of $V(G)$, it suffices to show that
\begin{equation} \label{23}
\sum_{1\le i< j\le p}e(V_i,V_j)> \left(k+\frac{\delta-1}{\delta}\right)(p-1).
\end{equation}
Let $r_i=e(V_i,V\setminus V_i )$ for $i=1,\dots, p$. Then \eqref{23} is equivalent to
\begin{equation}\label{24}
\sum_{i=1}^{p}r_i> 2\left(k+\frac{\delta-1}{\delta}\right)(p-1).
\end{equation}

Assume that $r_1\le \dots \le r_p$.
If $r_1\ge 2\left(k+\frac{\delta-1}{\delta}\right)$, then
\[
\sum_{i=1}^{p}r_i\ge 2\left(k+\frac{\delta-1}{\delta}\right)p>2\left(k+\frac{\delta-1}{\delta}\right)(p-1),
\]
so \eqref{24} follows.

Suppose that $r_1< 2\left(k+\frac{\delta-1}{\delta}\right)$.
Let $h$ be the maximal index among $1,\dots, p$ such that $r_h< 2\left(k+\frac{\delta-1}{\delta}\right)$.
Since $2\left(k+\frac{\delta-1}{\delta}\right)<2(k+1)\le \delta$, we have $r_i\le \delta-1$ for any $i\in\{1,\dots,h\}$. Then we have by Lemma \ref{U} that $|V_i|\ge \delta+1$ for any $i\in \{1,\dots,h\}$.

\begin{claim} \label{ca1}
$\sum_{i=1}^{h}r_i> 2\left(k+\frac{\delta-1}{\delta}\right)(h-1)$.
\end{claim}

\begin{proof} As $\lambda_2(G)<\delta$, $G$ is connected, so $r_1>0$. Claim \ref{ca1} follows trivially if $h=1$.

Suppose that $h\ge 2$. Let $i\in \{2,\dots,h\}$.
Let $B=A(G[V_1\cup V_i])$. By Lemma \ref{inter}, we have $\lambda_2(B)\le \lambda_2(G)$. Let $|V_1|=n_1$ and  $|V_i|=n_i$.
For $j=1,i$, let $a_j=\frac{\sum_{v\in V_j}d_G(v)-r_j}{n_j}$.
For $j=1,i$, since $n_j\ge \delta+1$ and $r_j< 2\left(k+\frac{\delta-1}{\delta}\right)$, we have the the assumption that
\[
a_j\ge \delta-\frac{r_j}{n_j}> \delta-\frac{2\left(k+\frac{\delta-1}{\delta}\right)}{\delta+1}>\lambda_2(G).
\]
Then, for $j=1,i$,
\[
a_j-\lambda_2(G)>\frac{2\left(k+\frac{\delta-1}{\delta}\right)}{\delta+1}-\frac{r_j}{n_j},
\]
and so
\begin{equation}\label{21}
a_j-\lambda_2(G)>\frac{2\left(k+\frac{\delta-1}{\delta}\right)}{\delta+1}-\frac{r_j}{\delta+1}.
\end{equation}

The quotient matrix $B'$ of $B$ with respect to the partition $V_1\cup V_i$ is
\[
B'=\begin{pmatrix}
a_1 & \frac{r}{n_1} \\\\
\frac{r_i}{n_i} & a_i
\end{pmatrix},
\]
where $r=e(V_1, V_i)$. By Lemma \ref{quo}, we have
\[
\lambda_2(B')\le \lambda_2(B)\le \lambda_2(G).
\]
By a direct calculation,
$\lambda_2(B')=\frac{1}{2}\left(a_1+a_i-\sqrt{(a_1-a_i)^2+\frac{4r^2}{n_1n_i}}\right)$, from which we have
\[
r^2=\left(a_1-\lambda_2(B')\right)\left(a_2-\lambda_2(B')\right)n_1n_i,
\]
so
\[
r^2\ge \left(a_1-\lambda_2(G)\right)\left(a_2-\lambda_2(G)\right)(\delta+1)^2.
\]
Now, from \eqref{21}, we have
\begin{align*}
r^2
&>\left(\frac{2\left(k+\frac{\delta-1}{\delta}\right)}{\delta+1}-\frac{r_1}{\delta+1}\right)
    \left(\frac{2\left(k+\frac{\delta-1}{\delta}\right)}{\delta+1}-\frac{r_i}{\delta+1}\right)
    \left(\delta+1\right)^2\\
&\ge\left(2\left(k+\frac{\delta-1}{\delta}\right)-r_i\right)^2,
\end{align*}
that is,
\[
e(V_1, V_i)+r_i> 2\left(k+\frac{\delta-1}{\delta}\right).
\]
So
\[
\sum_{i=2}^h e(V_1, V_i)+\sum_{i=2}^h r_i> 2\left(k+\frac{\delta-1}{\delta}\right)(h-1).
\]
Obviously, $r_1\ge \sum_{i=2}^h e(V_1, V_i)$. It follows that
\[
\sum_{i=1}^h r_i> 2\left(k+\frac{\delta-1}{\delta}\right)(h-1). \qedhere
\]
\end{proof}

By Claim \ref{ca1},
\[
\sum_{i=1}^{h}r_i> 2\left(k+\frac{\delta-1}{\delta}\right)(h-1).
\]
Thus
\begin{align*}
\sum_{i=1}^{p}r_i&=\sum_{i=1}^{h}r_i+\sum_{i=h+1}^{p}r_i\\
&> 2\left(k+\frac{\delta-1}{\delta}\right)(h-1)+2\left(k+\frac{\delta-1}{\delta}\right)(p-h)\\
&=2\left(k+\frac{\delta-1}{\delta}\right)(p-1),
\end{align*}
confirming  \eqref{24}.
\end{proof}


Given a graph $G$ of order $n$ and a real $\alpha\in [0,1)$, let $A_{\alpha}(G)=\alpha D(G)+(1-\alpha) A(G)$, where $D(G)$is the diagonal matrix of vertex degrees of $G$ \cite{GZ,NR}.
We denote by $\lambda_{\alpha, i}(G)$  the $i$-th largest eigenvalue of  $A_\alpha(G)$ of $G$ with $i=1,\dots, n$. Note that $\lambda_{0, i}(G)$ is the $i$-th largest eigenvalue of $G$,
$2\lambda_{\frac{1}{2}, i}(G)$ is the  $i$-th largest signless Laplacian eigenvalue of $G$.
Theorem \ref{s1} can be extended as the following form.

\begin{theorem}\label{s12}
Let $k$ be a positive integer, $\alpha\in [0,1)$, and let $G$ be a graph with minimum degree $\delta\ge 2k+2$.
Suppose that
\[
\lambda_{\alpha, 2}(G)<\delta-2\left(1-\alpha\right)\frac{k+\frac{\delta-1}{\delta}}{\delta+1},
\]
then $G$ has property $P(k, \delta)$.
\end{theorem}

\begin{proof}
Suppose that $G$ does not have  property $P(k, \delta)$. By Theorem \ref{s1}, we have $\lambda_2(G)\ge \delta-\frac{2\left(k+\frac{\delta-1}{\delta}\right)}{\delta+1}$.
By Weyl's inequalities (see \cite{So}), $\lambda_{\alpha, 2}(G)\ge \alpha\delta+(1-\alpha)\lambda_{2}(G)$, so
\begin{align*}
\lambda_{\alpha, 2}(G)
&\ge \alpha\delta+(1-\alpha)\left(\delta-\frac{2\left(k+\frac{\delta-1}{\delta}\right)}{\delta+1}\right)\\
& =\delta-2\left(1-\alpha\right)\frac{k+\frac{\delta-1}{\delta}}{\delta+1},
\end{align*}
which is a contradiction.
\end{proof}

\bigskip
\bigskip

\noindent {\bf Acknowledgement.}
This work was supported by the National Natural Science Foundation of China (No.~12071158).

\end{document}